\documentclass[11pt]{article}

\usepackage{setspace}
\usepackage{mathtools}
\usepackage{comment}
\usepackage[usenames,dvipsnames]{color}
\usepackage{amssymb,latexsym,mathrsfs,color,amsmath}
\usepackage[mathscr]{eucal}
\usepackage{xspace}
\usepackage{enumerate}
\usepackage{graphics}
\usepackage{graphicx}
\usepackage{amsthm}
\usepackage{enumitem}
\usepackage{mathrsfs}

\newlength{\temp}

\newcommand{\powb}[1]{\Pow\big({#1}\big)}

\newcommand{\procLocalTrash}{\ensuremath{\mathsf{LocalTrash}}\xspace}

\definecolor{mycol}{rgb}{0.32, 0.12, 0.32}
\newcommand{\mycol}[1]{\textcolor{mycol}{#1}}
\newcommand{\xO}{x_{_{\emptyset}}}

\newcommand{\closure}{\mathsf{P}}

\newcommand{\formProc}{\mathsf{formProcess}}

\newcommand{\mycomment}[1]{\footnotesize\textcolor{dark-gray}{- \hspace{-2.3pt}- \emph{#1}}}

\newcommand{\pumpCycle}{\mathsf{pumpCycle}}

\newcommand{\pChain}{\mathcal{P\hspace{-.8pt}C}}

\newcommand{\sT}{\mid}
\definecolor{dark-gray}{gray}{0.35}
\newcommand{\sets}{\{\:\mbox{\rm sets}\:\}}

\newcommand{\Vars}[1]{\mathrm{Vars}(#1)}

\newcommand{\VarsBig}[1]{\mathrm{Vars}\Big(#1\Big)}
\newcommand{\HF}{\textnormal{\textsf{HF}}\xspace}
\newcommand{\MLS}{\textnormal{\textsf{MLS}}\xspace}

\newcommand{\MLSPFIN}{\widehat{ \MLSP }}

\newcommand{\rkk}{\hbox{\sf rk}\xspace}
\newcommand{\ZFC}{\emph{ZFC}\xspace}

\newcommand{\setTheory}{\mathcal{S}}

\definecolor{grey}{rgb}{0.9, 0.9, 0.9}
\newlength{\savefboxrule}
\title{\textbf{Two Dichotomy Theorems}}

\author{Domenico Cantone \\
\emph{Dipartimento di Matematica e Informatica, Universit\`a di Catania}\\
\emph{Viale Andrea Doria 6, I-95125 Catania, Italy.}  \\
\mbox{E-mail:} \texttt{cantone@dmi.unict.it}
\\[1cm]
Pietro Ursino\\
\emph{Dipartimento di Fisica e Matematica, Universit\`a dell'Insubria}\\
\emph{Via Valleggio 11, I-22100 Como, Italy.}  \\
\mbox{E-mail:} \texttt{pietro.ursino@uninsubria.it}
}

\date{}
\newtheorem{mytheorem}{Theorem}[section] 
\newtheorem{mydef}[mytheorem]{Definition} 
\newtheorem{mylemma}[mytheorem]{Lemma} 
\newtheorem{mycorollary}[mytheorem]{Corollary} 
\newtheorem{myremark}[mytheorem]{Remark} 
\newtheorem{myconj}{Conjecture} 
\newtheorem{myfact}{Fact} 

\newcommand{\GE}{\mathrm{GE}}

\newcommand{\COMMENT}[1]{}

\newcommand{\inters}[2]{#1\cap #2\neq\emptyset}

\newcommand{\defAs}{\coloneqq}

\newcommand{\nats}{\omega}

\newcommand{\false}{{\bf f}}
\newcommand{\true}{{\bf t}}
\newcommand{\Not}{{\bf\neg}}
\renewcommand{\And}{\wedge}
\newcommand{\Or}{\vee}
\renewcommand{\implies}{\rightarrow}

\newcommand{\biimplies}{\leftrightarrow}

\newcommand{\Pow}{{\mathscr{P}}}
\newcommand{\pow}[1]{\Pow({#1})}

\newcommand{\powAst}{\Pow^{\ast}}
\newcommand{\powast}[1]{\powAst({#1})}

\newcommand{\rk}{\hbox{\sf rk}\;}

\newcommand{\Places}{\mathcal{P}}
\newcommand{\TARGETS}{\mathcal{T}}

\newcommand{\REDS}{\mathcal{R}}
\newcommand{\POWNODES}{\mathcal{Q}_{\sigma^*}}

\def\eod {{\unskip\nobreak\hfil\penalty50
\hskip2em\hbox{}\nobreak\hfil $\Box$
\parfillskip=0pt \finalhyphendemerits=0 \par \medskip}}

\def \N {{\mathbb N}}

\newcommand{\mypsi}{\Phi}
\newcommand{\boldcalV}{\mbox{\boldmath$\mathcal{V}$}}

\newcommand{\Finite}{\mathit{Finite}}

\begin{document}

\maketitle

\newcommand{\tbd}{\centerline{------TO BE DONE------}}
\newcommand{\MLSP}{\textnormal{\textsf{MLSP}}\xspace}
\newcommand{\MLSSP}{\textnormal{\textsf{MLSSP}}\xspace}
\newcommand{\MLSSPF}{\textnormal{\textsf{MLSSPF}}\xspace}
\newcommand{\MLSC}{\textnormal{\textsf{MLSC}}\xspace}
\newcommand{\MLSCNOTORD}{\textnormal{\textsf{MLSCNOTORD}}\xspace}
\newcommand{\MLSUC}{\textnormal{\textsf{MLSU\!$\times$}}\xspace}

\newcounter{instr}
\newcommand{\ninstr}{\refstepcounter{instr}\theinstr.}
\newcounter{instrb}
\newcommand{\ninstrb}{\refstepcounter{instrb}\textcolor{dark-gray}{\footnotesize{\theinstrb.}} \'}
\newcommand{\commentout}[1]{}
\newcommand{\minus}{\ensuremath{\mathsf{Minus}}\xspace}
\newcommand{\surplus}{\ensuremath{\mathsf{Surplus}}\xspace}
\newcommand{\MinusSurplus}{\ensuremath{\mathsf{Minus}\textsf{-}\mathsf{Surplus}}\xspace}
\newcommand{\DeltaMinus}{\Delta_{\mathrm{Minus}}}
\newcommand{\DeltaSurplus}{\Delta_{\mathrm{Surplus}}}

\newcommand{\mathcalW}[1]{\widehat{\mathcal{#1}}}


\section*{Introduction}
The core argument of Computable Set Theory is the problem of decidability of some fragments of Set Theory. The general features of this area of research have been widely settled in at least two books \cite{CFO89} and \cite{libro01}. A quite standard argument in order to set decidability of a language consists in proving a small model property that is to prove that whenever a formula of the language has model it has a model of a rank not exceeding a fixed number which depends only on the number of variables involved in the assigned formula.

In the following we deal with a particular behaviour of some set theoretic languages: the dichotomy property. This property says, roughly speaking, that whenever a formula of a fixed language is satisfiable it admits a model smaller than a fixed rank (depending only on the number of the variables of the formula) otherwise it admits an infinite model.
This property depends strongly on the expressivity of the language indeed the more a language has expressive power the more it is hard to close its models in a cage, which is the key to prove the decidability. Dichotomy is, in a sense, a medium point between small model property and undecidability since it does not assert that the model has a cage, but says that when this does not happen the model can diverge until the infinite. What it is forbidden is that admits unbound finite models without admitting an infinite model.

In the following we give a formal definition of the above discussed dichotomy property.

\noindent
Let $\mathfrak{F}$ be a fragment of set theory. For a set assignment $M$ and a formula $\varphi \in \mathfrak{F}$, we put
\[
\mathsf{dom}_{\varphi}(M) \defAs \bigcup \{Mx \sT x \in \Vars{\varphi}\}\/.
\]

\noindent
We say that a formula $\varphi \in \mathfrak{F}$ admits only finite models if for every set assignment $M$ we have
\[
M \models \varphi \qquad \Longrightarrow \qquad |\mathsf{dom}_{\varphi}(M)| < \nats \/.
\]

\begin{mydef}
A fragment $\mathfrak{F}$ of set theory is \emph{dichotomic} if there exists a map $f\colon \nats \rightarrow \nats$ (called \textsc{dichotomic map for $\mathfrak{F}$}) such that for every $\varphi \in \mathfrak{F}$ admitting only finite models and for every set assignment $M$ we have
\[
M \models \varphi \qquad \Longrightarrow \qquad \rkk\big(\mathsf{dom}_{\varphi}(M)\big) < f\big(|\Vars{\varphi}|\big)\/.
\]
\end{mydef}

The main results showed in the present paper are the following theorems

\begin{mytheorem}
The theory $\MLSP$ is dichotomic.
\end{mytheorem}

\begin{mytheorem}\label{notord}
The theory \MLSCNOTORD is not dichotomic.
\end{mytheorem}

and using this last result:

\begin{mytheorem}
  \MLSCNOTORD with disjoint unary union is undecidable
\end{mytheorem}

In order to deal with those kinds of problems, we largely make use of the theory of formative processes (see \cite{Urs06} , \cite{CU14} and \cite{CU17}). In particular we focus on the way in which it manages with the cycles of a P-graph, which can behave in a different way for different languages. Cycles are commonly seen as generators of infinite elements and they could be bounded in this production only by an external constraint, for example if one of the regions involved in the cycle is a component of an assignment $Mx$ of a variable $x$ which appears in the formula with literals as $Finite(x)$ or $x=\{y\}$. In this case, since the cycle has a cardinal bond, it cannot be pumped. In case of \MLSC or \MLSCNOTORD we are in presence of an internal bond as it can be seen in the proof of Theorem \ref{notord}. Indeed, still there is an infinite production of elements but they are just necessary to keep the "engine running". This phenomenon allows the creation of models of increasing rank never reaching the infinite, which in turns implies \MLSCNOTORD is not dichotomic.

\section{Syntax and semantics of the theory $\setTheory$}
\label{sse_MLSSPF}

The syntax of the quantifier-free fragment $\setTheory$ is defined as follows\footnote{By $\otimes$ we denote the unordered cartesian product.}.  The symbols of $\setTheory$ are:
\begin{itemize}
\item infinitely many set variables $x, y, z$, \ldots;

\item the constant symbol $\emptyset$;

\item the set operators $\cdot\cup\cdot$, $\cdot\cap\cdot$, $\cdot\setminus\cdot$, $\cdot\times\cdot$, $\cdot\otimes\cdot$,
$\{\cdot,\ldots,\cdot\}$, $\Pow(\cdot)$, $\bigcup (\cdot)$, $\bigcap (\cdot)$;

\item the set predicates $\cdot\subseteq\cdot$, $\cdot=\cdot$,
$\cdot\in\cdot$, $\Finite(\cdot)$.
\end{itemize}

\smallskip

The set of $\setTheory$-\textsc{terms} is the smallest collection of expressions such that:
\begin{itemize}
\item all variables and the constant $\emptyset$ are $\setTheory$-terms;

\item if $s$ and $t$ are $\setTheory$-terms, so are $s \cup t$, $s \cap
t$, $s \setminus t$, $s\times t$, $s\otimes t$, $\Pow(s)$, $\bigcup s$, and $\bigcap s$;

\item if $s_{1},\ldots,s_{n}$ are $\setTheory$-terms, so is $\{s_{1},\ldots,s_{n}\}$.
\end{itemize}

$\setTheory$-\textsc{atoms} have then the form
\[
s\subseteq t\/, \qquad s=t\/, \qquad s\in t\/, \qquad \Finite(s)\/,
\]
where $s, t$ are $\setTheory$-terms.

\begin{sloppypar}
$\setTheory$-\textsc{formulae} are propositional combinations of
$\setTheory$-atoms, by means of the usual logical connectives $\And$ (conjunction), $\Or$ (disjunction), $\Not$ (negation), $\implies$ (implication), $\biimplies$ (bi-implication), etc.
$\setTheory$-\textsc{lit\-er\-als} are $\setTheory$-atoms and their
negations.

For a $\setTheory$-formula $\mypsi$, we denote by $\Vars{\mypsi}$ the collection of set variables occurring in $\mypsi$ (similarly for $\setTheory$-terms).
\end{sloppypar}

\begin{table}[h]
\begin{center}
\begin{tabular}{ |l|c|c|c| }
\hline
\rule[-3.5mm]{0mm}{9mm}$\big\{x \cup \emptyset, \Pow\big(y \cap (z \setminus \{x\})\big)\big\} \setminus \bigcup (x \cap y)$ & $\setTheory$-term \\
\hline
\rule[-3.5mm]{0mm}{9mm}$\neg\Big(z \cup x \in \big\{x \cup \emptyset, \Pow\big(y \cap (z \setminus \{x\})\big)\big\} \rightarrow \big( x \notin \bigcap z \:\Or\: z \in \bigcup x \big)\Big)$ & $\setTheory$-formula\\
\hline\hline
\multicolumn{2}{|l|}{\rule[-3.5mm]{0mm}{9mm}$\VarsBig{\neg\Big(z \cup x \in \big\{x \cup \emptyset, \Pow\big(y \cap (z \setminus \{x\})\big)\big\} \rightarrow \big( x \notin \bigcap z \:\Or\: z \in \bigcup x \big)\Big)} = \{x,y,z\}$}\\
\hline
\end{tabular}
\end{center}
\caption{Some examples: an $\setTheory$-term, an $\setTheory$-formula, and the map $\Vars{\cdot}$.}
\end{table}

%
%

Our considerations will take place in a naive set theory which could be formalizable in the standard axiom system \ZFC, developed by Zermelo, Fraenkel, Skolem, and von Neumann (see \cite{Jec78}).  In particular, they will refer to the von Neumann standard cumulative
hierarchy of sets, a very specific model of \ZFC, and will assume the \textsc{Axiom of Regularity}.
Then semantics of $\setTheory$ is defined in the obvious way.  
A \textsc{set assignment} $M$ is any map from a collection $V$ of set variables (called the \textsc{domain of $M$}) into the universe $\boldcalV$ of all sets (in short, $M \in \boldcalV^{V}$ or $M \in \sets^{V}$).  Given set assignment $M$ over a collection of variables $V$, the \textsc{set domain of $M$} is the set $\bigcup M[V] = \bigcup_{v \in V}Mv$ and the \textsc{rank of $M$} is the ordinal
\[
\begin{array}{rcl}
\rk (M) & \defAs & \rk (\bigcup M[V])
\end{array}
\]
(so that, when $V$ is finite, $\rk (M) = \max_{v \in V} ~\rk (Mv)$).
A set assignment $M$ is \textsc{finite}, if so is its set domain.

\medskip

Let $M$ be a set assignment over a given collection $V$ of variables, and let $s,t,s_{1},\ldots,s_{n}$ be $\setTheory$-terms whose variables occur in $V$. We put, recursively,
\begin{align*}
M \emptyset & \defAs \emptyset\\
M (s \cup t) & \defAs Ms \cup Mt \\
M (s \cap t) & \defAs Ms \cap Mt \\
M (s \setminus t) & \defAs Ms \setminus Mt \\
M (s \times t) & \defAs Ms \times Mt \\
M (s \otimes t) & \defAs Ms \otimes Mt \\
M (\Pow(s)) & \defAs \Pow(Ms) \defAs \{ u \sT u \subseteq Ms \}\\
M \Big(\bigcup s\Big) & \defAs \bigcup Ms \defAs \{ u \sT u \in u' \in Ms \text{, for some } u'\}\\
M \Big(\bigcap s\Big) & \defAs \bigcap Ms \defAs \{ u \sT u \in u' \text{, for every } u'  \in Ms\} \quad \text{(provided that $Ms \neq \emptyset$)}\/.
\end{align*}
We also put
\begin{align*}
(s \in t)^{M} &\defAs \begin{cases}
\true & \text{if } Ms \in Mt\\
\false & \text{otherwise }
\end{cases}&
(s = t)^{M} &\defAs \begin{cases}
\true & \text{if } Ms = Mt\\
\false & \text{otherwise }
\end{cases}\\
(s \subseteq t)^{M} &\defAs \begin{cases}
\true & \text{if } Ms \subseteq Mt\\
\false & \text{otherwise }
\end{cases} &
(\Finite(s))^{M} &\defAs \begin{cases}
\true & \text{if $Ms$ is finite}\\
\false & \text{otherwise }
\end{cases}
\end{align*}
(where, plainly, $\true$ and $\false$ stand the truth-values \emph{true} and \emph{false}, respectively), and
\begin{align*}
(\Phi \:\And\: \Psi)^{M} &\defAs \Phi^{M} \:\And\: \Psi^{M} &
(\Phi \:\Or\: \Psi)^{M} &\defAs \Phi^{M} \:\Or\: \Psi^{M}\\
(\Phi \:\implies\: \Psi)^{M} &\defAs \Phi^{M} \:\implies\: \Psi^{M} &
(\neg \Phi)^{M} &\defAs \neg (\Phi^{M}) && \text{etc.}
\end{align*}
for all $\setTheory$-formulae $\Phi$, $\Psi$ such that $\Vars{\Phi}, \Vars{\Psi} \subseteq V$.

The set assignment $M$ is said to \textsc{satisfy} an $\setTheory$-formula $\mypsi$ if $\Phi^{M} = \true$ holds, in which case we also write $M \models \mypsi$ and say that $M$ is a \textsc{model} for $\mypsi$. If $\mypsi$ has a model, we say that $\mypsi$ is \textsc{satisfiable}; otherwise, we say that $\mypsi$ is \textsc{unsatisfiable}. If $\mypsi$ has a finite model, we say that it is \textsc{finitely satisfiable}. If $\mypsi$ has a model $M$ such that $Mx \neq My$ for all distinct variables $x,y \in \Vars{\mypsi}$, we say that it is \textsc{injectively satisfiable}. If $M' \models \mypsi$ for every set assignment $M'$ defined over $\Vars{\mypsi}$, then $\mypsi$ is said to be \textsc{true}. Two $\setTheory$-formulae $\mypsi$ and $\Psi$ are said to be \textsc{equisatisfiable} if $\mypsi$ is satisfiable if and only if so is $\Psi$.

\subsection{The decision problem for subtheories of $\setTheory$}\label{sse:reduction}
Let $\setTheory'$ be any subtheory of $\setTheory$. The \textsc{decision problem} (or \textsc{satisfiability problem}, or \textsc{satisfaction problem}) for $\setTheory'$ is the problem of establishing algorithmically whether any given $\setTheory'$-formula is satisfiable. If the decision problem for $\setTheory'$ is solvable, then $\setTheory'$ is said to be \textsc{decidable}. A \textsc{decision procedure} (or \textsc{satisfiability test}) for $\setTheory'$ is any algorithm which solves the decision problem for $\setTheory'$.
The \textsc{finite satisfiability problem} for $\setTheory'$ is the problem of establishing algorithmically whether any given $\setTheory'$-formula is finitely satisfiable. The \textsc{injective satisfiability problem} for $\setTheory'$ is the problem of establishing algorithmically whether any given $\setTheory'$-formula is injectively satisfiable.

\medskip

By making use of the disjunctive normal form, the satisfiability problem for $\setTheory$ can be  readily reduced to the same problem for conjunctions of $\setTheory$-literals. In addition, by suitably introducing fresh set variables to name subterms of the following types
\[
t_{1} \cup t_{2}, \quad t_{1} \cap t_{2}, \quad t_{1} \setminus t_{2}, \quad t_{1}\times t_{2}, \quad t_{1}\otimes t_{2}, \quad \pow{t}, \quad \bigcup t, \quad \bigcap t,
\]
where $t_{1},t_{2},t$ are $\setTheory$-terms, the satisfiability problem for $\setTheory$ can further be reduced to the satisfiability problem for conjunctions of $\setTheory$-literals of the following types
\begin{gather}\label{firstConj}
\begin{array}{l@{~~~~~~~~}l@{~~~~~~~~}l@{~~~~~~~~}l@{~~~~}l}
      x=y \cup z \/,  & x=y \cap z\/,  & x=y \setminus z\/,  & x=\{y_{1},\dots,y_{H}\}\/, & x=\pow{y} \/, \\
      x=\bigcup y \/, & x=\bigcap y \/, & x = y\/, & x \neq y \/, & s\times t \/, \\
      s\otimes t \/, & x\in y\/, & x\notin y\/, & x\subseteq y\/, & x\not\subseteq y\/,
\end{array}
\end{gather}
where $x,y,z, y_{1},\dots,y_{H}$ stand for set variables or the constant $\emptyset$.

Finally, by applying the following simplification rules
\begin{enumerate}[label=(s\arabic*)]
\item\label{simpl1} a literal of type $x = y$ can be replaced by the equivalent literal $x = y \cup y$,

\item\label{simpl2} a literal of type $x \not\subseteq y$ is equisatisfiable with $z' = x \setminus y \:\And\: z' \neq \emptyset$,

\item\label{simpl3} the constant $\emptyset$ can be eliminated by replacing it with a new variable $y_{_{\emptyset}}$ and adding the conjunct $y_{_{\emptyset}} = y_{_{\emptyset}} \setminus y_{_{\emptyset}}$,

\item\label{simpl4} a literal of type $x=y \cap z$ is equisatisfiable with $y' = y \setminus z \:\And\: x = y \setminus y'$,

\item\label{simpl5} a literal of type $x \subseteq y$ is equisatisfiable with $y = x \cup y$,

\item\label{simpl6} a literal of type $x \neq y$ is equisatisfiable with $x \in z' \:\And\: y \notin z'$,

\item\label{simpl7} a literal of type $x \notin y$ is equisatisfiable with $x \in z' \:\And\: z' = z' \setminus y$,
\end{enumerate}
(where $y'$ and $z'$ stand for fresh set variables), the satisfiability problem for $\setTheory$ can be reduced
%
%
%
to the satisfiability problem for conjunctions of $\setTheory$-atoms of the following types:
\begin{gather}\label{normConjT}
\begin{array}{c@{~~~~~~~~}c@{~~~~~~~~}c@{~~~~~~~~}c@{~~~~~~~~}c@{~~~~~~~~}c}
      x=y \cup z \/,  & x=y \setminus z\/,  & x\in y\/,  & x=\{y_{1},\dots,y_{H}\}\/,& s\otimes t \\
      s\times t \/, & x=\bigcup y \/, & x=\bigcap y \/, & x=\pow{y}
\end{array}
\end{gather}
(where $x\/,y\/,y_{1}\/,\dots,\/y_{H}$ stand for variables), which we call \textsc{normalized conjunctions} of $\setTheory$. Needless to say, working with normalized conjunctions simplifies the completeness and correctness proofs of decision procedures.

\smallskip

The reduction technique for the satisfiability problem to normalized conjunctions has been illustrated in the case of the whole theory $\setTheory$; however, it can directly be adapted to subtheories of $\setTheory$.


Notice that a satisfiability test for a subtheory $\setTheory'$ of $\setTheory$ can
also be used to decide whether any given formula $\mypsi$ in
$\setTheory'$ is \textsc{true}.
%
In fact, a formula $\mypsi$ is true if and only if its negation
$\neg\mypsi$ is unsatisfiable.

%
%

\subsubsection{Decidable fragments of set theory}\label{fragm}

Over the years, several subtheories of $\setTheory$ have been proved to be decidable. We mention in particular the theory Multi-Level Syllogistic (\MLS), which is the common kernel of most decidable fragments of set theory investigated in the field of Computable Set Theory. Specifically, \MLS is the propositional combination of atomic formulae of the following three types\footnote{As remarked above, intersection and set inclusion are easily expressible by means of literals of type (\ref{normMLS}).}
\begin{gather}\label{normMLS}
\begin{array}{c@{~~~~~~~~}c@{~~~~~~~~}c@{~~~~~~~~}c@{~~~~~~~~}c}
      x=y \cup z \/,  & x=y \setminus z\/,  & x\in y\/.
\end{array}
\end{gather}

Below we give the list of the subtheories of $\setTheory$ which we deal with in the present article:


\vspace{0.5cm}

\begin{tabular}{lll}
    \MLS: & $\cup$, $\cap$, $\setminus$, $\subseteq$,
    $=$, $\in$ & (cf. \cite{FOS80a})\\
  \MLSP: & $\cup$, $\cap$, $\setminus$,
    $\subseteq$, $=$, $\in$, $\mathscr{P}(\cdot)$
& (cf. \cite{CanFerSch85}) \\

  \MLSC: & $\cup$, $\cap$, $\setminus$, $\subseteq$,
    $=$, $\in$, $\times$ & (Open problem) \\
 \MLSCNOTORD: & $\cup$, $\cap$, $\setminus$, $\subseteq$,
    $=$, $\in$, $\otimes$ & (Open problem) \\
    \MLSSPF: & $\cup$, $\cap$, $\setminus$,
    $\subseteq$, $=$, $\in$, $\{\cdot\}$, $\mathscr{P}(\cdot)$, $\Finite(\cdot)$~~~
&
    (cf. \cite{CU14}) \\
     \textbf{\ldots} & \textbf{\ldots} & \textbf{\ldots}
\end{tabular}

\vspace{0.5cm}

\noindent The interested reader can find an extensive treatment of such results in \cite{CFO89} and \cite{libro01}.

\medskip

In several cases, the decidability of a given fragment of set theory has been shown by
proving for it a small model property.




\section{The theory $\MLSP$}

We proceed by describing some examples of the expressive power of $\MLSP$ , then we prove that it is dichotomic.

\subsection{Expressing hereditarily finite sets with $\MLSP$-conjunctions}\label{sectHF}

We show that it is possible to express hereditarily finite sets in the theory $\MLSP$, though the singleton operator is not a primitive operator of $\MLSP$.

For every set $s$, we plainly have
\[
\pow{s} = \bigcup_{s'' \subseteq s} \{ s''\}\,,
\]
where the singletons on the left-and-side are all pairwise disjoint.
Hence, in particular,

\begin{align}
\{ s \} &= \pow{s} \setminus \bigcup_{s'' \subsetneq s} \{ s''\} \notag\\
        &= \powb{ \bigcup_{s' \in s} \{ s' \} }  \setminus \bigcup_{s'' \subsetneq s} \{ s''\} \label{eq_dec}
\end{align}
From (\ref{eq_dec}) it follows that the singleton of a finite set $s$ can be expressed by a finite  expression involving the singletons of the elements and of the proper subsets of $s$, and the operators of binary union, set difference, and powerset.

The above observation suggests a way to represent each hereditarily finite set with a suitable formula of the theory $\MLSP$ involving only equality ($=$), conjunction ($\wedge$), the standard Boolean set operators $\cup$ and $\setminus$, and the powerset operator $\Pow$; thus, neither the membership relator nor the singleton operator are used.
We recall that \emph{hereditarily finite} sets are those sets that are {\em finite} and whose elements, elements of elements, etc., all are finite. We denote their collection by $\HF$. Plainly, $\HF = \mathcal{V}_{\omega}$, where $\omega = \{0,1,2,\ldots\}$ is the first infinite ordinal.\footnote{We are using here the well-known von Neumann encoding of integers, recursively defined as $0 := \emptyset$, and $i+1 := i\cup\{i\}$.}  Observe also that $\mathcal{V}_{n} \in \HF$, for every $n \in \omega$.

In our representation, we will use only variables of the form $x_{\{h\}}$, indexed by singletons of hereditarily finite sets. In addition, it will turn out that each representing formula $\varphi$ is satisfiable and also enjoys the following \emph{faithfulness condition}:
\begin{itemize}
\item if $M \models \varphi$, then for each variable $x_{\{h\}}$ occurring in $\varphi$ it must be the case that $M \, x_{\{h\}} = \{h\}$, where $h \in \HF$.
\end{itemize}

For each $h \in \HF$, we recursively define the representing $\MLSP$-formula $\varphi_{ \{ h \} }$ by putting
\begin{equation}\label{reprFormulae}
\varphi_{ \{ h \} } \defAs \begin{cases}
  \big( x_{\{ \emptyset \}} =  \powb{ x_{\{ \emptyset \}} \setminus x_{\{ \emptyset \}}} \big) & \text{if } h = \emptyset \\[.5cm]
  {\displaystyle \Bigg(\!x_{\{ h \}} = \powb{ \bigcup_{h' \in h} x_{\{ h' \}} }  \setminus \bigcup_{h'' \subsetneq h} x_{\{ h''\}}\!\Bigg) \:\wedge\: \bigwedge_{h' \in h} \varphi_{\{ h' \}} \:\wedge\: \bigwedge_{h'' \subsetneq h} \varphi_{\{ h''\}}} & \text{if } h \neq \emptyset\,.
\end{cases}
%
\end{equation}
To show that the recursive definition (\ref{reprFormulae}) is well-given, one has to exhibit a well-ordering $\prec$ of \HF such that, for $h \in \HF$, the following properties hold:
\begin{itemize}
\item[(P1)] $h' \prec h$, for every $h' \in h$, and

\item[(P2)] $h'' \prec h$, for every $h'' \subsetneq h$\,.
\end{itemize}
In particular any total ordering $\prec$ of \HF complying with the rank, i.e., such that
\[
\rk h' < \rk h \: \Longrightarrow \: h' \prec h\,,
\]
and extending the strict partial ordering $\subsetneq$ among sets of the same rank, i.e., such that
\[
[(\rk h' = \rk h) \wedge (h' \subsetneq h)] \: \Longrightarrow \: h' \prec h\,,
\]
satisfies (P1) and (P2) above. This is the case, for instance, for the ordering on \HF induced by the Ackermann encoding
  \begin{eqnarray*}
    \N(h) & = &  \sum_{h'\in h} \;2^{\N(h')}\,,
  \end{eqnarray*}
for $h \in \HF$.
Indeed, it is an easy matter to check that for each $h \in \HF$ we have
\begin{itemize}
\item $\N(h') < \N(h)$, for every $h' \in h$, and

\item $\N(h'') < \N(h)$, for every $h'' \subsetneq h$\,.
\end{itemize}

By induction on the Ackermann code $\N(h)$, it can easily be shown that each $\MLSP$-formula $\varphi_{ \{ h \} }$ is satisfiable and it also satisfies the above faithfulness condition. Hence, for every hereditarily finite set $h$, the formula
\[
x_{h} \in x_{ \{ h \}} \wedge \varphi_{ \{ h \} }
\]
faithfully expresses $h$ via the variable $x_{h}$.

Alternatively, a hereditarily finite set $h$ can also be expressed in $\MLSP$, via the variable $x_{h}$, as follows:
\[
\begin{cases}
x_{\emptyset} = x_{\emptyset} \setminus x_{\emptyset} & \text{if } h = \emptyset\\[.2cm]
\displaystyle \Big(x_{h} = \bigcup_{h' \in h} x_{\{ h' \}}\Big) \:\wedge\: \bigwedge_{h' \in h} \varphi_{\{ h' \}} & \text{otherwise}\,.
\end{cases}
\]

Finally, we observe that each $\MLSP$-formula $\varphi_{\{h\}}$ can be easily transformed into an equisatisfiable $\MLSP$-conjunction $\varphi'_{\{h\}}$ involving newly introduced variables.

\subsection{Expressing the operator $\powast{s_{1},\ldots,s_{k}}$ with $\MLSP$-conjunctions}\label{sectPowast}

Given a (possibly empty) finite list of sets $s_{1},\ldots,s_{k} \in \boldcalV$, $\powast{s_{1},\ldots,s_{k}}$ is the collection of the subsets of $s_{1} \cup \ldots \cup s_{k}$ which have non-null intersection with each of the sets $s_{1},\ldots,s_{k}$; in symbols,
\begin{equation}
    \label{powStar}
\powast{s_{1},\ldots,s_{k}} ~\defAs~ \Big\{\,s\subseteq\bigcup_{i=1}^{k} s_{i} \sT \inters{s}{s_{i}}\,, \text{ for } i=1,\ldots,k\,\Big\}\,.
\end{equation}
Thus, in particular, for the empty list $\Lambda$, we have $\powast{\Lambda} = \{\,\emptyset\,\}$. In addition, we have
\begin{equation}\label{eqPowastEmpty}
\powast{s_{1},\ldots,s_{k}} = \emptyset \: \Longleftrightarrow \: \bigvee_{i=1}^{k} s_{i} = \emptyset
\/.
\end{equation}
and
\begin{equation}
\label{powStarAsPow}
\powast{s_{1},s_{2}} = \Pow\left(s_{1} \cup s_{2}\right) \setminus \big(\Pow\left(s_{1} \setminus s_{2}\right) \cup \Pow\left(s_{2} \setminus s_{1}\right)\big)\,.
\end{equation}
The equivalence (\ref{eqPowastEmpty}) follows immediately from the very definition of $\powast{\{s_{1},\ldots,s_{k}\}}$. Concerning (\ref{powStarAsPow}), if $s \in \powast{s_{1},s_{2}}$, then plainly $s \subseteq s_{1} \cup s_{2}$, $s \not\subseteq s_{1} \setminus s_{2}$, and $s \not\subseteq s_{2} \setminus s_{1}$, so that $s \in \Pow\left(s_{1} \cup s_{2}\right) \setminus \big(\Pow\left(s_{1} \setminus s_{2}\right) \cup \Pow\left(s_{2} \setminus s_{1}\right)\big)$. Conversely, if $s \in \Pow\left(s_{1} \cup s_{2}\right) \setminus \big(\Pow\left(s_{1} \setminus s_{2}\right) \cup \Pow\left(s_{2} \setminus s_{1}\right)\big)$, then $s \subseteq s_{1} \setminus s_{2}$ and $s \cap s_{1} \neq \emptyset \neq s \cap s_{2}$, which yields, by (\ref{powStar}), $s \in \powast{s_{1},s_{2}}$. Hence, (\ref{powStarAsPow}) follows.

Equation (\ref{powStarAsPow}) readily generalizes to
\begin{equation}
\label{powStarAsPowGen}
\powast{s_{1},\ldots,s_{k}} = \Pow\left(\mathcal{S}\right) \setminus \left(\bigcup_{i=1}^{k} \Pow\left(\mathcal{S} \setminus s_{i}\right)\right)\,,
\end{equation}
where $\mathcal{S} \defAs \bigcup_{i=1}^{k}s_{i}$.

From (\ref{powStarAsPowGen}) it follows immediately that a literal of the form $x = \powast{y_{1},\ldots,y_{k}}$ can be expressed by the $\MLSP$-literal
\[
x = \Pow\left(\textstyle\bigcup_{i=1}^{k} y_{i}\right) \setminus \left(\bigcup_{j=1}^{k} \Pow\left(\textstyle\left(\bigcup_{i=1}^{k} y_{i}\right) \setminus y_{j}\right)\right)\,,
\]
which, as before, can in turn be transformed into an equisatisfiable $\MLSP$-conjunction involving newly introduced variables.
%

\subsection{The subtheory $\MLSPFIN$}
We investigate the subtheory of $\MLSP$, denoted $\MLSPFIN$, consisting of the collection of $\MLSP$-formulae which do not admit infinite models.

We provide the following semantic definition for $\MLSPFIN$.
\begin{mydef}\label{def_MLSPhat}
$\MLSPFIN := \big\{ \psi \in \MLSP \sT \psi \wedge (\bigvee_{x \in \mathrm{Vars}(\psi)} \neg \Finite(x)) \text{ \emph{is unsatisfiable}} \big\}$.\eod
\end{mydef}

At a first glance it seems that arguing from the fact that $\MLSP$ has a decidable satisfiability problem (cf.\ \cite{CanFerSch85}) we can assert that the satisfiability problem for $\MLSPFIN$ is decidable, since $\MLSPFIN\subseteq \MLSP$. Unfortunately the problem of deciding if a formula $\psi$ is or is not in $\MLSPFIN$ is not a trivial one therefore using the decision procedure of $\MLSP$ you could find out that it is, for example, satisfiable but the procedure do not tell you if $\psi$ is inside $\MLSPFIN$ or not. We summarize the preceding remark by asserting that the real decision procedure for $\MLSPFIN$ is advocated by the decision procedure for \MLSSPF (namely, $\MLSSP$ with the finiteness predicate $\Finite(\cdot)$; cf.\ \cite{CU14}).
\bigskip

Let  $\psi$ be an $\MLSP$-formula. We say that $\psi$ satisfies the \emph{ordinal condition} if
\begin{enumerate}[label=(\alph*)]
\item[($\star$)]
\quad $\big\{ \rkk(Mx) \sT x \in \Vars{\psi} \big\}$ is an ordinal, for every model $M$ of $\psi$.
\end{enumerate}
Observe that when $\big\{ \rkk(Mx) \sT x \in \Vars{\psi} \big\}$ is an ordinal, then $\big\{ \rkk(Mx) \sT x \in \Vars{\psi} \big\} = \rkk(M)$, and conversely.

\smallskip

If the ordinal condition holds for $\psi$, then $\psi \in \MLSPFIN$, since then $\rkk(M) < |\Vars{\psi}|$ holds for every model $M$ of $\psi$, and therefore
\[
\psi \wedge (\textstyle\bigvee_{x \in \mathrm{Vars}(\psi)} \neg \Finite(x))
\]
is unsatisfiable.

We claim that the ordinal condition ($\star$) is also necessary for $\psi \in \MLSPFIN$ to hold, i.e.:
\begin{myconj}\label{mainConj}
If $\psi \in \MLSPFIN$, then $\big\{ \rkk(Mx) \sT x \in \Vars{\psi} \big\}$ is an ordinal,
for every model $M$ of $\psi$.
\end{myconj}

\section[$\MLSP$ is dichotomic]{$\MLSP$ is dichotomic}\label{Powerset}
The following theorem proves Conjecture \ref{mainConj}.
\begin{mytheorem}
  Let $\varphi\in MLSP$ a satisfiable formula, $M$ a model of $\varphi$ with related transitive partition $\Sigma_l$ with formative process $\varpi = \big\langle
\!\big(\Sigma_i\big)_{i\leqslant l}, (\bullet), \TARGETS,\POWNODES \big\rangle$ longer than $k=2^{2^{\vert Var(\varphi)\vert}}$. Then there exists an infinite model of $\varphi$.
\end{mytheorem}
\begin{proof}
It can easily be checked that $\Sigma_l$ has $P$-graph with a cycle, $\mathcal C$, which the formative process follows at least two times.
Consider the following set of variables of $\varphi$: $$\mathcal{V}_{\cal C}=\{x\in Var(\varphi)\mid \exists \sigma\in {\mathcal C}\wedge\sigma\subseteq x \}$$
Fix $x\in\mathcal{V}_{\cal C}$.
Define a new formula  $\varphi '$: $$\varphi'=\varphi\wedge\neg Finite(x) $$

Observe that $\varphi '\in\MLSSPF$. In order to prove our thesis it is sufficient to show that $\mathcal{C}$ is a pumping cycle for $\sigma$ (see Section 5.4.2).
First we need to change our formative process in a colored formative process. With this aim in mind we consider the following procedure:

\newpage

\begin{quote}{\small
\begin{tabbing}
xx \= xx \= xx \= xx \= xx \= xx \= xx \= xx \= xx \= xx \kill
\hspace{-10pt}\textbf{procedure} $\procLocalTrash(\formProc\,,\, \pumpCycle)$;\\
\> \mycomment{$\formProc$ is a
$\Places$-process $\varpi = \big\langle
\!\big(\Sigma_\mu\big)_{\mu\leqslant\xi}, (\bullet), \TARGETS,
\POWNODES \big\rangle$ and $\pumpCycle$ is a }\\
\> \mycomment{
$\mathcal{C} = \langle
\sigma_0,A_0,\sigma_1,A_1\dotsb \sigma_m,A_m=A_0 \rangle$.}\\[3pt]
\>
LT := $\mathcal{C}_{places}$\\
Label "unchecked" all nodes\\
\>  loop (until there is a node unchecked) pick a node $B$ unchecked such that $B\cap LT\neq\emptyset$ \textbf{do}\\

\>  \> i:=LC(B)\\
\>  \> pick $\sigma$ such that $\Delta^{(i)}(\sigma_{i+1})\neq\emptyset$\\
\>  \> put $\sigma$ inside LT\\
\>  \> label $B$ checked\\
\>  \textbf{return} $LT$;\\
\hspace{-10pt}\textbf{end procedure};
\end{tabbing}
}
\end{quote}

The procedure above takes as input a finite formative process together with a simple cycle of the P-graph of the resulting transitive partition and returns a a collection of places $LT$.

Consider now the collection $LT$ of places and the following colored $\Places$-process
    $\varpi = \big\langle\!\big(\Sigma_i\big)_{i\leqslant l}, (\bullet), \TARGETS,\REDS,\POWNODES \big\rangle$ where $\REDS = \Places\setminus LT$.

By a simple check of procedure \procLocalTrash we can deduce that $LT=\closure$ is a $\Pow$-closed set of places, i.e., $LT
    \subseteq \Places \setminus \REDS$ and every $\Pow$-node $B \in
    \POWNODES$ that intersects $\closure$ has a local trash in it.
    Indeed, pick a node $B$ such that $B\cap LT\neq\emptyset$  by procedure Local Trash there is a $\sigma '\in LT$ which receives elements in the Last Call of $B$.
     We show that such a $\sigma '$ is inside LT and it is a local trash for $B$. The former is clearly fulfilled arguing from the construction of procedure Local Trash, the latter depends on the fact that all the nodes distribute all their elements, therefore for all $A$ such that $\sigma\in A$ $LC(A)\ge LC(B)$.
Since the cycle is repeated two times for all $B \in\POWNODES$ that intersects $\closure$ $\GE (B)\ge i_0$, where $i_0$ is the step in which the cycle is finished to be done.
We can summarize the above results just saying that $\pChain = \big\langle %
	\mathcal{C}, %
	i_0, %
	\sigma, %
	LT
\big\rangle$ is a pumping chain for $\sigma$.

\end{proof}

\section{The theories \MLSC and \MLSCNOTORD}

Regarding syntax and semantics of \MLSC and \MLSCNOTORD we refers to Section \ref{fragm}.

\noindent
We proceed by describing some examples of expressive power of \MLSC .

\noindent
We recall that by $\times$ we mean ordered cartesian product \'a la Kuratowski.

\subsection[Expressing hereditarily finite sets with \MLSC-conjunctions]{Expressing hereditarily finite sets with \MLSC-conjunctions}

We show that the singleton operator is expressible in the theory \MLSC in such a way to preserve rank-boundedness. Besides we show that it is possible to express hereditarily finite sets, as well.

Consider the following conjuncts in the first column (in the second column we indicate some deductions):
\begin{align*}
(\alpha) && &x \in x' \in w \in z = y' \times y'  & &\mycol{ |x'|  \leq 2 }\\
(\beta) && &x' \in y' \in w & &\mycol{ x' = \{x\} \text{ ~and~ } y' = \{x, \{x\}\} }\\
&&&& &\mycol{ w = \big\{\{x\}, \{x, \{x\}\}\big\} = (x,\{x\}) }\\
&&&& &\mycol{ z =  \{x, \{x\}\} \times  \{x, \{x\}\} }
\end{align*}

Let $\psi$ be a formula of \MLS extended with the Cartesian product, $x$ a variable occurring in $\psi$, and $x',y',w,z$ variables not occurring in $\psi$, and consider the formula
\[
\psi' \defAs \psi \And \alpha \And \beta\/.
\]
Then for each model $M$ of $\psi$, there is a unique extension $M'$ of $M$ over the variables $x',y',w,z$ such that $M'$ satisfies $\psi'$. In such a model we have
\begin{align*}
M'x' &= \{Mx\}\\
M'y' &= \{Mx, \{Mx\}\}\\
M'w  &= \big\{\{Mx\}, \{Mx, \{Mx\}\}\big\}\\
M'z  &= \{Mx, \{Mx\}\} \times  \{Mx, \{Mx\}\}\/.
\end{align*}
Besides, if the model $M$ is rank-bounded, $M'$ is so, as well.

Thus, the singleton operator can be expressed in \MLS plus Cartesian product in such a way to preserve rank-boundedness.

\bigskip

The following lemma summarizes the above discussion in a more concise way.

\begin{mylemma}\label{singl}
The singleton operator is expressible in \MLSC.
\end{mylemma}
\begin{proof}
The conjunction of the following literals expresses the literal $y = \{x\}$\/:
\[
x,x' \in y'\/, \quad
x' \neq x\/, \quad
y' \in y' \times y'\/, \quad
x \in y \subsetneq y'\/,
\]
where $x'$ and $y'$ are fresh variables.\footnote{It is not hard to see that the literals $x' \neq x$ and $y \subsetneq y'$ are expressible by conjunctions in \MLSC .}
\end{proof}

Since singletons are expressible, and also the empty set is readily expressible by the \MLS-literal $\xO = \xO \setminus \xO$, it follows immediately that every hereditarily finite set is expressible in \MLS extended with the Cartesian product (by rank-bounded formulae).
\subsection[\MLSCNOTORD is not dichotomic]{\MLSCNOTORD is not dichotomic}\label{Cartesian}
We recall that by $\otimes$ we mean not ordered cartesian product.

\begin{mytheorem}\label{lemdich}
Let $Y,Z$ be sets satisfying the following two conditions
\[
\{\emptyset\} \otimes (\{\emptyset\} \cup Y) = Z\/, \qquad Y \subsetneq Z\/.
\]
Then $\rkk(Z) < \omega$ and $|Z \setminus Y| = 1$.
\end{mytheorem}
\begin{proof}
Let $Y,Z$ be sets satisfying the conditions of the theorem.

Let us define recursively the sequence of sets $\{a_{n}\}_{n \in \omega}$, where
\begin{align*}
a_{0} &\defAs \{\emptyset\}\\
a_{n+1} &\defAs \{ \emptyset, a_{n} \}\/,
\end{align*}
and put $A \defAs \{ a_{n} \sT n \in \omega \}$.

To begin with, we show that $Z \subseteq A$. We proceed by contradiction. If  $Z \not\subseteq A$, let $z \in Z \setminus A$ of minimal rank. Since $z \in \{\emptyset\} \otimes (\{\emptyset\} \cup Y)$, then $z = \{\emptyset, y\}$, for some $y \in \{\emptyset\} \cup Y$. In fact $y \in Y$, since otherwise we would have $z = \{\emptyset\} \in A$.
Since $\rkk(y) < \rkk(z)$ and $y \in Z$, we have $y \in A$, so that $y = a_{k}$ for some $k \geq 0$. But then we would have $z = a_{k+1} \in A$, a contradiction.

Next we show that, for every $k \geq 0$,
\begin{equation}\label{myEq}
a_{k} \in Z \qquad \Longrightarrow \qquad a_{\ell} \in Y\/,  \text{ for every } 0 \leq \ell < k\/.
\end{equation}
We proceed again by contradiction, and let $\overline{k} \in \omega$ be minimal such that (\ref{myEq}) does not hold for $a_{\overline{k}}$. But then $a_{\overline{k}} = \{\emptyset, a_{\overline{k}-1}\}$, with $a_{\overline{k}-1} \in Y$. Hence, by the minimality of $\overline{k}$ it follows that $a_{k'} \in Y$, for all $1 \leq k' < \overline{k}-1$, so that (\ref{myEq}) holds for $a_{\overline{k}}$, a contradiction.

From (\ref{myEq}) it follows that $|Z \setminus Y| = 1$. Indeed, if $a_{k_{1}}, a_{k_{2}} \in Z \setminus Y$, with $k_{1} \neq k_{2}$, then, by (\ref{myEq}), it would follow that $\min (a_{k_{1}}, a_{k_{2}}) \in Y$, a contradiction.

Finally, let $Z \setminus Y = \{a_{\overline{k}}\}$. From what we have shown, it follows that
\[
Z = \{a_{i} \sT 0 \leq i \leq \overline{k} \}\/,
\]
and therefore $\rkk(Z) = \overline{k}+2 < \omega$, concluding the proof of the theorem.
\end{proof}
\begin{mycorollary}
The theory \MLSCNOTORD is not dichotomic.
\end{mycorollary}
\begin{proof}

By Lemma \ref{singl} we can consider the following \MLSCNOTORD-formula
\[
\varphi(y,z) \defAs \{\emptyset\} \otimes (\{\emptyset\} \cup y) = z \: \And \: y \subsetneq z\/.
\]
By Lemma \ref{lemdich} for every set assignment $M$, we have $M \models \varphi(y,z)$ if and only if there exists a $k \in \nats$ such that
\begin{align*}
My &= \{a_{\ell} \sT \ell < k\}\\
Mz &= \{a_{\ell} \sT \ell \leq k\}\/.
\end{align*}
Thus, $\varphi(y,z)$ admits only finite models, yet we have
\[
\sup \{ \rkk(Mz) \sT M \models \varphi(y,z)\} = \nats\/.
\]
\end{proof}

\section[\MLSCNOTORD with disjoint unary union is undecidable]{\MLSCNOTORD with disjoint unary union is undecidable}

For a set $S$, we define the disjoint unary union $\biguplus S$ by putting
\[
\biguplus S \defAs \big\{ t \sT (\exists ! s)  (s \in S \: \And \: t \in s) \big\}\/.
\]

We shall show that the collection $\mathfrak{C}$ of conjunctions of literals of the following types
\begin{align}
\label{eqn:literals}
\begin{aligned}
x & \in y\/, & x &= y \cup z\/, \qquad & x &= y \setminus z\\
x & = \biguplus y\/, \qquad & x &= y \otimes z
\end{aligned}
\end{align}
has an undecidable satisfiability problem.

In view of the results in \cite{CanCutPol90}, it is sufficient to show that we can express the following \emph{positive} literals
\[
|x| = |y|\/, \qquad \Finite(x)
\]
with conjunctions in $\mathfrak{C}$. This will be done in Facts~\ref{factCard} and \ref{factFinite} below, respectively.

The following fact states that positive literals of type $|x| = |y|$ are expressible with conjunctions in $\mathfrak{C}$.
\begin{myfact}\label{factCard}
Let $\varphi$ be any conjunction in $\mathfrak{C}$, possibly involving the variables $x$ and $y$, and let $y'$ and $z$ be any variables distinct from $x$ and $y$ and not occurring in $\varphi$. Let $C(x,y,y',z)$ be the conjunction of the following three literals
\[
y' = \{x\} \otimes y\/, \qquad z \subseteq x \otimes y'\/, \qquad \biguplus z = x \cup y'\/.
\]
Then the formulae $\varphi \wedge |x| = |y|$ and $\varphi \: \And \: C(x,y,y',z)$ are equisatisfiable.
\end{myfact}
\begin{proof}
We observe that Lemma \ref{singl} holds for \MLSCNOTORD , as well.
The thesis follows from the semantics of the operator $\biguplus$ and the validity of the following implication
\[
y' = \{x\} \otimes y \quad \Longrightarrow \quad \big( x \cap y' = \emptyset \:\And\: |y| = |y'|\big)\/.
\]
\end{proof}

The following fact states that positive literals of type $\Finite(x)$ are expressible by conjunctions in $\mathfrak{C}$.
\begin{myfact}\label{factFinite}
Let $\varphi$ be any conjunction in $\mathfrak{C}$, possibly involving the variable $x$, and let $w$, $y$, and $z$ be any variables distinct from $x$ and not occurring in $\varphi$. Let $F(x,w,y,z)$ be the conjunction of the following three literals
\[
z = \{\emptyset\} \otimes (\{\emptyset\} \cup w)\/, \qquad w \subsetneq z\/, \qquad |x| = |w|\/.
\]
Then the formulae $\varphi \And \Finite(x)$ and $\varphi \And F(x,w,y,z)$ are equisatisfiable.
\end{myfact}
\begin{proof}
The thesis follows directly from Theorem~1 in Section \ref{Cartesian}.
\end{proof}

\begin{myremark}
Fact~\ref{factCard} could be restated by using the standard unary union operator $\bigcup$ in place of the disjoint unary union operator $\biguplus$, but adopting the multiset semantics, rather than the standard semantics of set theory.

Yet another alternative approach to prove Fact~\ref{factCard} would be to express literals of type \linebreak$|x| = |y|$ by means of the standard unary union operator $\bigcup$ and of the predicate $\mathsf{isPartition}(\Sigma)$ expressing that $\Sigma$ is a partition, namely a collection of pairwise disjoint sets.

Finally, another alternative approach to reduce Tenth's Hilbert problem (for an extensive treatise of this problem see \cite{Mat93}) to a fragment of set theory would be to extend \MLSC with \emph{just one} occurrence of the predicate $\mathsf{isPartitionOf}(\Sigma,x)$, expressing that $\Sigma$ is a partition of $x$. Indeed, it is not hard to see that \MLSC with a single \emph{positive} conjunct of type $|x| = |y|$ allows one to encode any Diophantine equation.
\end{myremark}

%

\bibliographystyle{alpha-Springer}
\end{document}